\documentclass[a4paper,12pt]{amsart}
\usepackage{amssymb,amscd,amsmath,latexsym}
\usepackage[dvips]{graphicx}

\calclayout

\theoremstyle{plain}

\newtheorem{theo}{Theorem}[section]
\newtheorem{prop}[theo]{Proposition}
\newtheorem{lemm}[theo]{Lemma}
\newtheorem{coro}[theo]{Corollary}

\theoremstyle{definition}
\newtheorem*{defi}{Definition}
\newtheorem{exam}[theo]{Example}

\theoremstyle{remark}

\numberwithin{equation}{section}

\newcommand{\field}[1]{\mathbb{#1}}
\newcommand{\C}{\field{C}}
\newcommand{\R}{\field{R}}
\newcommand{\Q}{\field{Q}}
\newcommand{\Z}{\field{Z}}
\renewcommand{\H}{\field{H}}

\renewcommand{\P}{\mathcal P}

\def\T{\mathcal T}

\DeclareMathOperator{\Hom}{Hom}

\DeclareMathOperator{\Lie}{Lie}
\DeclareMathOperator{\Area}{Area}
\DeclareMathOperator{\vol}{vol}
\DeclareMathOperator{\codim}{codim}
\DeclareMathOperator{\Int}{Int}
\DeclareMathOperator{\Td}{Td}

\def\1{\underline{\C}}
\def\Cn{(\C^*)^n}
\def\K{K}
\def\a{\mathbf a}
\def\T{{T_\C}}
\def\m{\mu}
\def\V{\mathcal V}
\def\XR{X_\R}
\def\MR{M_\R}
\def\TR{T_\R}

\begin{document}

\title{Toric Topology}

\author[]{Mikiya Masuda}
\address{Graduate School of Science, Osaka City University}

\maketitle

\section{Introduction}

Toric geometry is a bridge connecting algebraic geometry and combinatorics.  Through this bridge, some problems and results in combinatorics can be interpreted in terms of algebraic geometry and vice versa.  For example, counting lattice points in a lattice convex polytope and Kouchnirenko-Bernstein's theorem describing the relation between the number of solutions of a system of Laurent polynomials and Minkowski's mixed volume of the Newton polytopes associated with the equations are such a problem or a result (see \cite{fult93}, \cite{oda88}, \cite{kh-ka08}).  
Among them, the solution of McMullen's conjecture (\cite{mcmu71}) by Stanley \cite{stan80} must have been striking.  McMullen's conjecture says that the face numbers of simplicial polytopes would be characterized by three conditions.  Stanley proved the necessity of these three conditions by applying the Poincar\'e duality theorem (topology), the hard Lefschetz theorem (algebraic geometry) and Macaulay's theorem (commutative algebra) to the projective toric orbifolds associated with the simplicial polytopes.  On the other hand, Billera-Lee \cite{bi-le80} constructed simplicial polytopes satisfying those three conditions almost at the same time, so that McMullen's conjecture was affirmatively solved dramatically.  It was around 1980.  This result is now called the $g$-theorem (see \cite{fult93}, \cite{hibi95}). 

The theory of toric geometry can be developed using topology instead of algebraic geometry to some extent (\cite{da-ja91}, \cite{bu-pa02}, \cite{ha-ma03}, \cite{masu99}).  Namely, one can construct a bridge between topology and combinatorics.  Since the method is topological, the geometrical object treated there need not be algebraic varieties.  Moreover, the combinatorial object which appears is a wider class than that in toric geometry and one can see a natural generalization of results on convex polytopes in combinatorics to the wider class.  One can also see the difference between algebraic geometry and topology through the generalization.  In this article, the author will explain toric topology from this point of view.  Nigel Ray is probably the first person who initiated the word toric topology but he uses this word in a wider meaning than used here.  We restrict our concern to smooth manifolds or orbifolds in this article but Ray uses the word toric topology for the study of not only manifolds or orbifolds but also topological spaces with symmetry (especially torus actions) which fit well to combinatorics.        

Toric topology is an emerging field as a crossroad of several fields.  Ray expresses the relations among these fields in terms of {\it Toric Tetrahedron} (see Figure 1 below).  The four vertices A, C, S, T of the tetrahedron represent four fields Algebraic Geometry，Combinatorics, Symplectic Geometry, Topology respectively and the edges and faces show the relations among the four fields.  The location of the content of this article is an interior point near the edge joining the vertices C and T and close to the vertex T.  So, this article is not comprehensive.  
We refer the reader to the book  (\cite{bu-pa02}) by Buchstaber-Panov，the paper  (\cite{pa-ra08}) by Panov-Ray，the lecture note (\cite{buch08}) by Buchstaber for other viewpoints of toric topology, and to the proceedings \lq\lq Toric Topology" (Harada et al., editors, Contemp. Math. 
460 (2008)), especially the paper \lq\lq An invitation to toric topology: A vertex four of a remarkable 
tetrahedron" by Buchstaber-Ray for a general survey of toric topology.  The article \cite{masu05-1} explains a part of the content of this article to general readers.  

\begin{figure}[h]
\begin{center}
\setlength{\unitlength}{0.6mm}
\begin{picture}(55,60)(10,-10)

\put(0,-20){\includegraphics[scale=0.2]{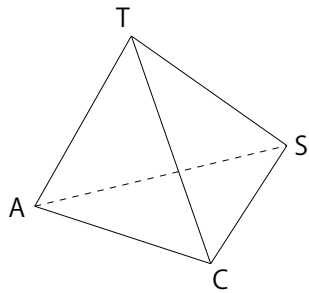}}

\end{picture}
\end{center}
\caption{Toric Tetrahedron}
\end{figure}

\section{Toric geometry}

We review the fundamental theorem in toric geometry.  The standard textbooks 
(\cite{dani78}, \cite{ewal96}, \cite{fult93}, \cite{ishi00}, \cite{oda88}) in toric geometry define the notion of fan and then introduce toric varieties associated with fans, but we take the reverse approach for our later purpose.  In the following, algebraic varieties are considered over the field of complex numbers $\C$ and $\C^*=\C\backslash\{0\}$.

\begin{defi} 
A normal algebraic variety $X$ of complex dimension $n$ with an effective algebraic action of $(\C^*)^n$ is called a {\it toric variety} if the action has an open dense orbit.   
\end{defi}

The point is that the dimension of the acting group is same as the dimension of the space.  There is only one open dense orbit in $X$ which can be identified with the group $(\C^*)^n$ and the other orbits are finitely many and have smaller dimensions.  We often do not mention the action of $(\C^*)^n$ on $X$ explicitly.

$(\C^*)^n$ itself or complex $n$-dimensional faithful representation spaces of $(\C^*)^n$ are toric varieties.  Cartesian products of finitely many toric varieties and quotients of toric varieties by finite subgroups of $(\C^*)^n$ are again toric varieties.  A typical example of a compact smooth toric variety is the complex projective space $\C P^n$.  It is a simple example but useful to understand the general argument, so we shall give it below.  
Henceforth we call a compact smooth toric variety a \emph{toric manifold}.    

\begin{exam} \label{CPn}
We define an action of an element $(g_1,\dots,g_n)$ of $(\C^*)^n$ on $\C P^n$ by 
\[
[z_1,\dots,z_n,z_{n+1}]\mapsto [g_1z_1,\dots,g_nz_n,z_{n+1}]
\]
where $[z_1,\dots,z_n,z_{n+1}]$ denotes the homogeneous coordinate of $\C P^n$.  
$\C P^n$ with this action of $(\C^*)^n$ is a toric manifold. The open dense orbit is the set of points with all the coordinates non-zero.  We denote by $X_i$ the submanifold of $\C P^n$ defined by $z_i=0$ for $i=1,\dots,n+1$.  They have the following properties. 
\begin{enumerate}
\item[(1)] $X_i$'s intersect transversally and the union of $X_i$'s is the complement of the open dense orbit in $\C P^n$,  
\item[(2)] $X_i$ is the closure of an orbit of dimension $n-1$ and fixed pointwise under some $\C^*$-subgroup $T_i$ of $(\C^*)^n$.  In fact, $T_i$ is the $\C^*$-subgroup of $(\C^*)^n$ with the identity except the $i$-coordinate for $1\le i\le n$, and the diagonal $\C^*$-subgroup of $(\C^*)^n$ for $i=n+1$.  
\end{enumerate}
\end{exam}

Here is one more example of toric manifolds. 

\begin{exam}[Bott tower] \label{Bott}
A sequence of $\C P^1$-fiber bundles 
\[
B_n\stackrel{}\longrightarrow B_{n-1}
\stackrel{}\longrightarrow 
\dots \stackrel{}\longrightarrow B_1
\stackrel{}\longrightarrow 
B_0=\{ \text{a point}\}
\]
is called a {\it Bott tower} of height $n$ (\cite{gr-ka94}) and we call the top manifold $B_n$ a {\it Bott manifold}.  Precisely speaking, the $\C P^1$-fibration is the projectivization of the Whitney sum of a complex line bundle $L_k$ and the trivial line bundle over $B_{k-1}$. If 
$B_{k-1}$ admits an action of $(\C^*)^{k-1}$, then the action lifts to that on $L_k$ so that $L_k$ becomes an equivariant line bundle.  This action of $(\C^*)^{k-1}$ together with scalar multiplication of $\C^*$ on fibers of $L_k$ produces an action of    
$(\C^*)^k$ on $B_k$, so that $B_k$ becomes a toric manifold.  $B_2$ is a Hirzebruch surface and if all the line bundles $L_k$ are trivial, then $B_n=(\C P^1)^n$.  

Complex line bundles over a space $B$ bijectively correspond to $H^2(B;\Z)$ through the first Chern class.  Since 
$H^2(B_{k-1};\Z)$ is isomorphic to $\Z^{k-1}$, Bott towers (without actions) can be parameterized by $\bigoplus_{k=1}^n\Z^{k-1}$.  We note that it happens that even if two Bott towers are different, their Bott manifolds can be diffeomorphic or homeomorphic, in other words, a Bott manifold can have different Bott tower structures.  
\end{exam}

We shall review the definition of fans.  Let $N$ be a free abelian group of rank $n$. Then 
$N_\R=N\otimes_\Z\R$ is a real vector space of dimension $n$.  If a cone $\sigma$ in $N_\R$ spanned by a finitely many elements $v_1,\dots,v_k$ in $N$ 
\[
\sigma=\{ r_1v_1+\dots+r_kv_k\mid r_1\ge 0,\dots,r_k\ge 0\}
\]
satisfies $\sigma\cap(-\sigma)=\{0\}$, then it is called \emph{strongly convex rational polyhedral cone}.  If the generators $v_1,\dots,v_k$ are linearly independent over $\Q$, $\sigma$ is called \emph{simplicial}, more strongly, if the generators are a part of a basis of $N$, then $\sigma$ is called \emph{non-singular}.  Faces of $\sigma$ can naturally be defined.  A \emph{fan} in $N$ is a collection of finitely many strongly convex rational polyhedral cones in $N_\R$ satisfying these two conditions: 
\begin{enumerate}
\item[(1)] every face of a cone in $\Delta$ is also a cone in $\Delta$, 
\item[(2)] the intersection of two cones in $\Delta$ is a face of each.  
\end{enumerate}
The dimension of the fan $\Delta$ in $N$ is defined to be the dimension of $N_\R$, that is $n$.  
If every cone in $\Delta$ is non-singular (resp. simplicial), then $\Delta$ is called \emph{non-singular} (resp. \emph{simplicial}). 
If the union of cones in $\Delta$ is the entire space $N_\R$, then the fan $\Delta$ is called \emph{complete}. 
The fans in Figure~\ref{CP2fan} are both non-singular and of dimension 2.  The fan (1) is not complete while the fan (2) is complete. 


\begin{figure}[h]
\begin{center}
\setlength{\unitlength}{0.6mm}
\begin{picture}(55,60)(10,-10)
\put(-40,0){\includegraphics[scale=1.02]{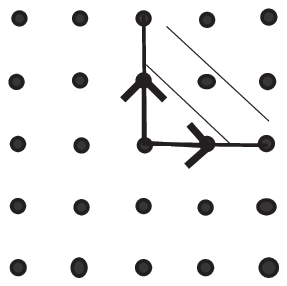}}
\put(30,0){\includegraphics[scale=1.02]{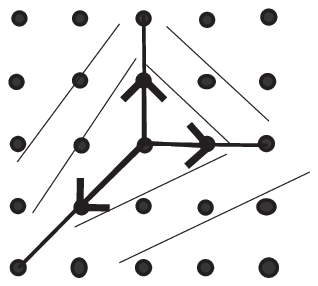}}
\put(3,22){$v_1$}
\put(-20,43){$v_2$}
\put(73,22){$v_1$}
\put(50,43){$v_2$}
\put(41,20){$v_3$}
\put(-15,-10){(1)}
\put(55,-10){(2)}
\end{picture}
\end{center}
\caption{}\label{CP2fan}
\end{figure}


\begin{theo}[Fundamental theorem in toric geometry] \label{fund}
There is a bijective correspondence between isomorphism classes of toric varieties of complex dimension $n$ and fans in $N_\R$ of real dimension $n$.
\end{theo}

More precisely speaking, morphism can be defined among toric varieties and fans respectively and there is a functor between the category of toric varieties and the category of fans giving equivalence as categories. 
Table 1 shows some corresponding notions between toric varieties and fans.  

\begin{table}[hbt]
\begin{tabular}{|c|c|} \hline
toric variety & fan \\\hline\hline
compact & complete \\ \hline
orbifold & simplicial \\ \hline
smooth & non-singular \\ \hline
Euler characteristic &  the number of cones of dim $n$ \\\hline
\end{tabular} \vspace{1.25ex}
\caption{A part of dictionary between toric varieties and fans}
\label{dictionary}
\end{table}

We shall give the bijective correspondence in Theorem~\ref{fund}.  As mentioned at the beginning of this section, one usually associates a toric variety with a fan but we shall give the reverse correspondence for toric manifolds (i.e. compact smooth toric varieties).    

Let $X$ be a toric manifold of complex dimension $n$.  Any $X$ has a unique open dense which is isomorphic to $(\C^*)^n$ and the other orbits are finitely many and have smaller dimensions.  Therefore, those orbits of smaller dimensions and their neighborhoods should characterize $X$.  We consider closures of codimension one orbits, denoted $X_1,\dots,X_m$.  Each $X_i$ is a closed connected submanifold of codimension one and they intersect transversally (see Example~\ref{CPn}).  We call $X_i$ a \emph{characteristic submanifold} of $X$.  They are invariant divisors in terms of algebraic geometry.   We will derive two data from them.  In the following $m$ denotes the number of characteristic submanifolds of $X$ and $[m]=\{1,\dots,m\}$. 

\medskip
{\bf Datum 1.}  We consider a family of subsets in $[m]$ 
\[
\K_X:=\{ I\subset [m]\mid \cap_{i\in I}X_i\not=\emptyset\}
\]
recording which $X_i$'s intersect.  Clearly, $K_X$ is an (abstract) simplicial complex.  
Since $X_i$'s intersect transversally, the cardinality of an element in $K_X$ is at most $n$.  On the other hand, $K_X$ has an element of cardinality $n$ because $X$ is compact.   Therefore, the simplicial complex $K_X$ is of dimension $n-1$.  
When $X$ is $\C P^n$ in Example~\ref{CPn}, $\K_X$ is isomorphic to the boundary complex of an $n$-simplex.  When $X$ is $B_n$ in Example~\ref{Bott}, $K_X$ is isomorphic to the join of $n$ number of $S^0$.  

\medskip
{\bf Datum 2.} The set $\Hom(\C^*,\Cn)$ of homomorphisms from $\C^*$ to $(\C^*)^n$ forms an abelian group under the multiplication of $(\C^*)^n$.  The correspondence 
\[
\begin{split}
\Z^n\quad&\to \Hom(\C^*,\Cn)\\
(a_1,\dots,a_n)&\mapsto (z\mapsto (z^{a_1},\dots,z^{a_n}))
\end{split}
\]
gives an isomorphism between $\Z^n$ and $\Hom(\C^*,(\C^*)^n)$.  We denote an element of $\Hom(\C^*,(\C^*)^n)$ corresponding to $\a\in\Z^n$ by $\lambda_{\mathbf a}$ through the isomorphism.  Remember that $X_i$ is fixed pointwise under some $\C^*$-subgroup $T_i$ of $\Cn$.  Therefore, a primitive element $\a_i\in\Z^n$ such that $\lambda_{\a_i}(\C^*)=T_i$ will be determined up to sign.  The problem of which primitive elements we should choose can be solved as follows.  We look at the normal bundle $\nu_i$ of $X_i$.  Since $\nu_i$ is a complex line bundle and $T_i$ fixes the base space $X_i$ of the normal bundle, the action of $T_i$ on $\nu_i$ defined via differential preserves fibers of $\nu_i$.   In fact, 
$\lambda_{\a_i}(g)\in T_i$ ($g\in \C^*$) acts on $\nu_i$ as scalar multiplication by $g$ or $g^{-1}$ on the fibers.  Therefore, if we require that 
the action of $\lambda_{\a_i}(g)$ on $\nu_i$ is just the multiplication by $g$, then $\a_i$ is uniquely determined without ambiguity of sign. 

Using the two data $\K_X$ and $\{\a_1,\dots,\a_m\}$, we define a fan.  To each element $I$ in $\K_X$, we form a cone 
$\angle \a_I:=\{\sum_{i\in I}r_i\a_i\mid r_i\ge 0\}$ in $\Z^n\otimes \R=\R^n$.  The collection of these cones 
$$\Delta_X:=\{\angle \a_I\mid I\in K_X\}$$
is the fan corresponding to the toric manifold $X$.  We may think of the fan $\Delta_X$ as visualization  of the above two data.

\begin{exam} \label{CPnfan}
As for Example~\ref{CPn}, we have $m=n+1$ and $K_X$ consists of all proper subsets of $[n+1]$.  The vector $\a_i$ for $0\le i\le n$ is the $i$th fundamental vector in $\Z^n$ and $\a_{n+1}$ is the vector $(-1,\dots,-1)$.  Therefore, we obtain the fan (2) in Figure~\ref{CP2fan} when $n=2$. 
\end{exam}

\bigskip

\section{Equivariant cohomology}

Equivariant cohomology was introduced by A. Borel (\cite{bore60}) around 1960, so it is often called Borel cohomology.  Equivariant cohomology is a quite useful tool in the study of transformation groups.  In fact, known results such as Smith fixed point theorem were reproved or improved using equivariant cohomology, and further applications of equivariant cohomology were discovered in 1960's and 1970's (\cite{bred72}), \cite{hsia75}).  In 1980's, it was recognized by Atiyah-Bott (\cite{at-bo84}) and Guillemin-Sternberg (\cite{gu-st99}) that equivariant cohomology fits well to the study of transformation groups in symplectic geometry, and it is recently used in the study of Mirror symmetry and Schubert calculus.  In these applications, the localization theorem (see \cite{al-pu93}, \cite{hsia75}) plays an essential role.  

In this section, we will see that equivariant cohomology fits well to the correspondence discussed in the previous section.  In the following, $X$ will be a toric manifold as in the previous section.  We set $\T=(\C^*)^n$ for simplicity and denote a universal principal $\T$-bundle by $E\T\to B\T$, where $E\T$ is a contractible space with a free $\T$-action and $B\T$ is its orbit space.  Explicitly, $E\T=(\C^\infty\backslash\{0\})^n$ and $B\T=(\C P^\infty)^n$.  We consider the diagonal $\T$-action on $E\T\times X$ and the ordinary cohomology of its orbit space $E\T\times_{\T} X$ is called the equivariant cohomology of $X$, that is, 
\[
H^*_{\T}(X):=H^*(E\T\times_{\T} X)
\]
is the equivariant cohomology of $X$.  Equivariant cohomology looks complicated at first glance, but its computation is actually easier than ordinary cohomology by virtue of the localization theorem.   

Remember that the characteristic submanifolds $X_i$'s characterize the toric manifold $X$.  We shall observe that they are closely related to the structure of $H^*_\T(X)$.  

\begin{defi} 
For $i=1,\dots,m$ the inclusion map from $X_i$ to $X$ induces the equivariant Gysin homomorphism $H^q_{\T}(X_i)\to H^{q+2}_{\T}(X)$ and we denote by $\tau_i\in H^2_\T(X)$ the image of $1\in H^0_{\T}(X_i)$ by the homomorphism.  The $\tau_i$ can be thought of as the equivariant Poincar\'e dual to the cycle $X_i$ in $X$.  
\end{defi}

Since the characteristic submanifolds $X_i$ intersect transversally, $\prod_{i\in I}\tau_i$ for $I\in [m]$ is the equivariant Poincar\'e dual of   $\bigcap_{i\in I}X_i$.  In particular, $\prod_{i\in I}\tau_i=0$ when $\bigcap_{i\in I}X_i=\emptyset$ (i.e. $I\notin K_X$).  It turns out that $H^*_\T(X)$ is generated by $\tau_i$'s and there is no other relations among $\tau_i$'s.  Namely we have   

\begin{prop} \label{ring}
$H^*_\T(X)=\Z[\tau_1,\dots,\tau_m]/(\prod_{i\in I}\tau_i \mid I\notin K_X)$ as rings. 
\end{prop}

This proposition says that $H^*_\T(X)$ is the face ring (or Stanley-Reisner ring) of the simplicial complex $K_X$.  
In fact, the information of $H^*_\T(X)$ as a ring is equivalent to that of the simplicial complex $K_X$. 

The projection from $E\T\times X$ onto the first factor $E\T$ induces a fiber bundle 
\begin{equation} \label{fiber}
X\stackrel{\iota}\longrightarrow E\T\times_\T X \stackrel{\pi}\longrightarrow E\T/\T=B\T. 
\end{equation}
To see $H^*_\T(X)$ as a ring means to see the total space of the fibration and overlooks how this fiber bundle is twisted, in other words, does not look at the projection $\pi$.  However, $H^*_\T(X)$ is not only a ring but also an algebra over $H^*(B\T)$ through $\pi^*\colon H^*(B\T)\to H^*_\T(X)$.  This algebra structure must catch how the fiber bundle is twisted.  Since $H^*(B\T)$ is a polynomial ring generated by $H^2(B\T)$, it suffice to see the image of $H^2(B\T)$ by $\pi^*$ to know the algebra structure of $H^*_\T(X)$ over $H^*(B\T)$.  The following proposition describes the image.     
  
\begin{prop} \label{vi}
To each $i=1,\dots,m$, there is a unique $v_i\in H_2(B\T)$ such that the identity
\[
\pi^*(u)=\sum_{i=1}^m\langle u,v_i\rangle \tau_i \quad(\forall u\in H^2(B\T))
\]
holds, where $\langle\ ,\ \rangle$ denotes the paring between cohomology and homology.  
\end{prop}

\begin{proof} 
The proof is easy if we admit Proposition~\ref{ring}.  By Proposition~\ref{ring}, $H^2_\T(M)$ is a free abelian group generated by $\tau_1,\dots,\tau_m$, so one can express $\pi^*(u)=\sum_{i=1}^mv_i(u)\tau_i$ with a unique integer $v_i(u)$ depending on $u$ for each $i$.  Since $\pi^*$is a homomorphism, $v_i(u)$ is linear with respect to $u$.  Therefore, $v_i$ can be regarded as an element of $H_2(B\T)$ dual to $H^2(B\T)$  so that  
$v_i(u)=\langle u,v_i\rangle$.
\end{proof}

A homomorphism from $\C^*$ to $\T$ induces a continuous map $B\C^*\to B\T$ between their classifying spaces and since $B\C^*=\C P^\infty$, $H_2(B\C^*)$ is an infinite cyclic group.  Therefore, once we fix a generator of $H_2(B\C^*)$, we obtain an element of $H_2(B\T)$ as the image of the fixed generator by the homomorphism induced by the continuous map. This gives an isomorphism  
\begin{equation} \label{HomST}
H_2(B\T)\cong\Hom(\C^*,\T). 
\end{equation}
We will denote by $\lambda_v$ the element in $\Hom(\C,\T)$ corresponding to $v\in H_2(B\T)$ through the isomorphism above.  

\begin{lemm}
$\lambda_{v_i}(\C^*)$ is the $\C^*$-subgroup which fixes $X_i$ pointwise.  To be more precise, if $\a_i$ is the vector in $\Z^n$ introduced in Datum 2 in the previous section, then $\lambda_{\a_i}=\lambda_{v_i}$.  Namely, if we identify $H_2(B\T)$ with $\Z^n$, then $v_i=\a_i$.  
\end{lemm}

Proposition~\ref{vi} and the lemma above tell us that the algebra structure of $H^*_\T(X)$ over $H^*(B\T)$ contains Datum 2 in the previous section.  This together with Proposition~\ref{ring} shows that one can reproduce the fan of $X$ from $H^*_\T(X)$.  This argument also says that the free abelian group $N$ used to define a fan should be $H_2(B\T)$ if we use equivariant cohomology or view toric geometry from a topological point of view.  

Since $H^{\rm{odd}}(B\T)=0$ and $X$ is compact and smooth, we have $H^{\rm{odd}}(X)=0$.  Therefore the Leray-Serre spectral sequence of the fiber bundle \eqref{fiber} collapses, so 
\begin{equation} \label{free}
\text{$H^*_\T(X)=H^*(B\T)\otimes H^*(X)$ as a module over $H^*(B\T)$.}
\end{equation}
This implies that the restriction map $\iota^*\colon H^*_\T(X)\to H^*(X)$ is surjective and its kernel is the ideal generated by the image of 
$H^{>0}(B\T)$ by $\pi^*$.  Moreover, since $H^*(B\T)$ is a polynomial ring generated by elements in $H^2(B\T)$, we obtain the following theorem from Propositions~\ref{ring} and \ref{vi}. 

\begin{theo}[Danilov-Jurkiewicz] \label{tcring}
We set $\iota^*(\tau_i)=\mu_i$.  Then 
$H^*(X)=\Z[\mu_1,\dots,\mu_m]/\mathcal I$.  Here $\mathcal I$ is the ideal generated by the following two types of elements: 
\begin{enumerate}
\item[(1)] $\prod_{i\in I}\mu_i \quad (I\notin K_X)$,
\item[(2)] $\sum_{i=1}^m\langle u,v_i\rangle \mu_i \quad(\forall u\in H^2(B\T))$. 
\end{enumerate}
\end{theo}

It becomes clear why the two types of elements above appear if we use equivariant cohomology.   

We said that we can reproduce the fan of $X$ from $H^*_\T(X)$, but this is slightly misleading because the specified elements $\tau_i$'s are used in Propositions~\ref{ring} and \ref{vi}.  It is not obvious that one can still reproduce if we regard $H^*_\T(X)$ as an abstract algebra over $H^*(B\T)$.  Namely, the question is whether two toric manifolds $X$ and $X'$ are isomorphic if $H^*_\T(X)$ and $H^*_\T(X')$ are isomorphic as algebras over $H^*(B\T)$, but it turns out that this is almost the case (\cite{masu08}).

\section{Torus manifold}

The story in Sections 2 and 3 for toric geometry can be developed in the category of topology to some extent.  Since the compact torus $T=(S^1)^n$   is a deformation retract of $\T=\Cn$, the equivariant cohomology $H^*_T(X)$ of $X$ with the restricted action of $T$ is isomorphic to $H^*_\T(X)$.  Therefore, the argument developed in the previous section works for $T$-actions.     

In the following, we will treat only compact and smooth manifolds for simplicity but the argument below works for non-compact manifolds or orbifolds with suitable modification.  Let $M$ be a smooth orientable closed manifold of dimension $2n$ with an effective smooth action of $T=(S^1)^n$ having a fixed point.  The fixed point set consists of finitely many isolated points.  We call such a manifold $M$ a \emph{torus manifold} (\cite{ha-ma03})\footnote{In \cite{ha-ma03}, an omniorientation introduced later is incorporated in the definition of torus manifold.}. 
A toric manifold (i.e. a compact smooth toric variety) with the restricted action of $T$ is a torus manifold but there are many torus manifolds which are not toric manifolds.  

\begin{exam} \label{exam12}
Let $(g_1,\dots,g_n)$ be an element of $T$.   

(1) The $2n$-dimensional unit sphere $S^{2n}$ in $\C^n\times\R$ with the $T$-action defined by 
\[
(z_1,\dots,z_n,y)\to (g_1z_1,\dots.g_nz_n,y)
\]
is a torus manifold but this is not a toric manifold when $n\ge 2$. 

(2) The orbit space of the standard complex $n$-dimensional representation of $T$ 
\[
(z_1,\dots,z_n)\in \C^n \to (g_1z_1,\dots,g_nz_n)\in \C^n
\]
can be identified with the orthant $(\R_{\ge 0})^n$ of $\R^n$ by taking the modulus of the coordinates of $\C^n$.  
Any faithful complex $n$-dimensional representation of $T$ is obtained from the standard representation by composing an automorphism of $T$, so its orbit space can also be identified with the orthant $(\R_{\ge 0})^n$.  Therefore, if a torus manifold $M$ is locally equivariantly homeomorphic to a complex $n$-dimensional faithful representation space of $T$, then the orbit space $M/T$ becomes a manifold with corners and $M$ is called a \emph{quasitoric manifold} if the orbit space $M/T$ is a simple convex polytope (see Section 5)\footnote{The notion of quasitoric manifold was introduced by Davis-Januszkiewicz
(\cite{da-ja91}) and they used the word \lq\lq toric manifold" for quasitoric manifold.  But the word toric manifold was already used in algebraic geometry as smooth toric variety, so Buchstaber-Panov(\cite{bu-pa02}) started using the word quasitoric manifold to avoid confusion. }.  
Davis-Januszkiewicz (\cite{da-ja91}) have developed a theory similar to toric geometry for quasitoric manifolds using topological technique.  Any toric manifold of complex dimension less than or equal to 3 is a quasitoric manifold but it is not known whether there is a toric manifold which is not quasitoric manifold, in other words, whether the orbit space of a toric manifold by the restricted compact torus action is a simple polytope as a manifold with corners.  On the other hand, there are many quasitoric manifolds which are not toric manifolds. For example,   $\C P^2\#\C P^2$ is a quasitoric manifold but not a toric manifold because it does not admit a complex (even almost complex) structure.\footnote{Added in the English translation:  Recently, the notion of a \emph{topological toric manifold} has been introduced in \cite{ih-fu-ma10}.  It is a closed smooth manifold of dimension $2n$ with an effective smooth action of $(\C^*)^n$ which is locally equivariantly diffeomorphic to sum of complex one-dimensional \emph{smooth} representation spaces of $(\C^*)^n$.  The family of topological toric manifolds contains both toric manifolds and quasitoric manifolds.}    
\end{exam}

We have assigned a fan to a toric manifold in Section 2 but the argument works even if we replace $\T$ by $T$, so it turns out that one can assign a combinatorial object like a fan to a torus manifold.  We shall explain this in more details.  
A torus manifold $M$ has finitely many closed codimension 2 submanifolds fixed pointwise under some $S^1$-subgroups.  We denote them by  
$M_i$ $(i=1,\dots,m)$.  When $M$ is a toric manifold $X$, these $M_i$ are nothing but the characteristic submanifolds $X_i$.  Because of this reason, we will call $M_i$'s the characteristic submanifolds of $M$.  Similarly to Datum 1 in Section 2, we obtain an (abstract) simplicial complex 
\[
K_M:=\{ I\subset [m]\mid \cap_{i\in I}M_i\not=\emptyset\}. 
\]

In order to obtain Datum 2 in Section 2 for $M$, we need to take orientations on $M$ and $M_i$'s into account.  Since $M$ is assumed to be orientable and each $M_i$ is fixed pointwise under some $S^1$-subgroup of $T$, $M_i$ is also orientable.  A choice of orientations on $M$ and $M_i$'s is called an \emph{omniorientation} on $M$.  We fix an omniorientation on $M$.  It determines a compatible orientation on the normal bundle $\nu_i$ of $M_i$.  Since $\nu_i$ is of real dimension 2, the compatible orientation (together with a $T$-invariant fiber metric) on $\nu_i$ defines a complex structure on $\nu_i$.  Therefore, an element   
$v_i\in H_2(BT)$ can be associated to $M_i$ by a similar argument in Sections 2 and 3.  The element $\lambda_{v_i}$ in $\Hom(S^1,T)$ corresponding to $v_i$ satisfies the following: 
\begin{enumerate}
\item[(1)] $\lambda_{v_i}(S^1)$ fixes $M_i$ pointwise, 
\item[(2)] $\lambda_{v_i}(g)_*(\xi)=g\xi$ $(g\in S^1,\ \xi\in \nu_i)$, 
\end{enumerate}
where $\lambda_{v_i}(g)_*$ in (2) above denotes the differential and $g\xi$ the complex scalar multiplication by $g$ on $\xi$. 

Using these two data $K_M$ and $\{v_i\}_{i=1}^m$, we form cones like we did in Section 2.  Then the cones can overlap as is shown in Figure~\ref{degree2}, where the two dimensional cones are spanned by $v_i$ and $v_{i+1}$ $(i=1,\dots,7)$（$v_8=v_1$） and they are all non-singular.  

\vspace{0.5cm}

\begin{figure}[h]
\begin{center}
\setlength{\unitlength}{1.0mm}
\begin{picture}(55,30)(5,10)
\put(0,10){\includegraphics[scale=1.7]{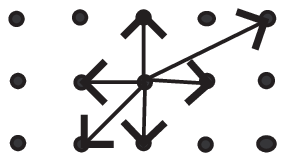}}
\put(43,22){\large{$v_1=v_4$}}
\put(31,43){\large{$v_2$}}
\put(12,12){\large{$v_3$}}
\put(55,40){\large{$v_5$}}
\put(9,31){\large{$v_6$}}
\put(31,13){\large{$v_7$}}
\end{picture}
\end{center}
\caption{}\label{degree2}
\end{figure}

In the case of torus manifolds, we derive one more datum. The simplicial complex $K_M$ is of dimension $n-1$ and if $I$ is a simplex of dimension $n-1$ in $K_M$, then $I$ consists of $n$ elements and $M_I:=\bigcap_{i\in I}M_i$ consists of finitely many $T$-fixed points.  For each $p\in M_I$, the tangent space $\tau_pM$ of $M$ at $p$ decomposes as follows: 
\[
\tau_pM=\bigoplus_{i\in I}\nu_i|_p
\]
where $\nu_i|_p$ denotes the restriction of $\nu_i$ to the point $p$.  Therefore $\tau_pM$ have two orientations: one is the orientation determined from the orientation on $M$ and the other is the orientation determined from the right hand side of the above identity.  
We define $\epsilon(p)=+1$ when these orientations agree and $\epsilon(p)=-1$ otherwise, and then define 
\[
\begin{split}
w^+(I):&=\#\{p\in M_I\mid \epsilon(p)=+1\}\\
w^-(I):&=\#\{p\in M_I\mid \epsilon(p)=-1\}\\ 
w(I):&=w^+(I)-w^-(I)=\sum_{p\in M_I}\epsilon(p).
\end{split}
\]
When $M$ is a toric manifold, $M$ consists of one point and $w(I)=1$.  However, $M_I$ can be more than one point and $w(I)$ can be different from $1$ for a general torus manifold $M$.  For example, if  $M=S^{2n}$ $(n\ge 2)$, then $M_I$ consists of two points and $w^+(I)=w^-(I)=1$ and hence $w(I)=0$.  

We call the triple $\Delta_M:=(K_M,\{v_i\}_{i=1}^m,w^\pm)$ the \emph{multi-fan} of the omnioriented torus manifold $M$.  When we form cones using $K_M$ and $\{v_i\}_{i=1}^m$, 
the functions $w^\pm$ and $w$ assign integers (or weights) to each cone of maximal dimension $n$.  In the case where $M$ is a toric manifold, $w^+=1$ and $w^-=0$ for every $n$-dimensional cone so that $w=1$.  

The cones in the multi-fan $\Delta_M$ do not overlap randomly.  The following theorem implies that the overlapping degree of the cones is controlled by the Todd genus of $M$.  

\begin{theo}[\cite{masu99}] \label{theo:Todd}
Let $v$ be an element of $H_2(BT)$ which is not contained in any cone of dimension $n-1$ in the multi-fan $\Delta_M$.  If $M$ admits a $T$-invariant weakly complex structure\footnote{This means that the Whitney sum of the tangent bundle of $M$ and a trivial vector bundle of certain dimension admits a $T$-invariant complex structure.}, then 
\begin{equation} \label{Todd}
\text{the Todd genus of $M$}=\sum_{I\in K_M, s.t.\ v\in \angle v_I}w(I).
\end{equation}
\end{theo}

The right hand side in \eqref{Todd} is the sum of the integers $w(I)$ over $I$ whose associated $n$-dimensional cone $\angle v_I$ contains $v$.  This sum seems to depend on the choice of $v$ but is actually independent of the choice of $v$ because so is the left hand side of \eqref{Todd}.  When $M$ is a toric manifold, the Todd genus of $M$ is $1$ and $w(I)=1$ (more precisely $w^+(I)=1,w^-(I)=0$).  Therefore, \eqref{Todd} implies that when $M$ is a toric manifold, the cones in the fan of $M$ have no overlap and the union of the cones is the entire space $H_2(BT;\R)$.  As for a torus manifold $M$ with Figure~\ref{degree2} as the multi-fan, $w(I)=1$ and the Todd genus of $M$ is 2 for a suitable choice of an omniorientation on $M$.  

The correspondence from torus manifolds to multi-fans is not one to one unlike the toric case.  It happens that different torus manifolds can associate the same multi-fan and the characterization of the multi-fans obtained from torus manifolds is not known.  However, as is seen in Theorem~\ref{theo:Todd}, some topological invariants of a torus manifold $M$ (such as signature, more generally $T_y$ or $\chi_y$ genus and elliptic genus)  can be described in terms of the multi-fan $\Delta_M$ (\cite{hatt06}, \cite{hatt08}, \cite{ha-ma03}, \cite{ha-ma05}).  As an application, one can prove the following. 

\begin{theo}[\cite{ha-ma05}]
If the first Chern class of a toric manifold $X$ of complex dimension $n$ is divisible by $N$, then 
$N\le n+1$.  Moreover, when $N=n+1$, $X$ is isomorphic to $\C P^n$, and when $N=n$, $X$ is isomorphic to a projective bundle of a Whitney sum of some complex line bundles over $\C P^1$.\footnote{When a toric manifold  $X$ is projective, this theorem can be obtained from a standard argument in algebraic geometry (\cite{fuji05}).  The paper \cite{fuji05} also treats the case where $X$ has singularity.}
\end{theo}

If a toric manifold $X$ is projective, then the orbit space $X/T$ can be identified with a moment map image of $X$ so that $X/T$ is a simple convex polytope.  Unless $X$ is projective, it is not known whether $X/T$ is still a simplex convex polytope but $X/T$ is a manifold with corners and all faces of $X/T$ (even $X/T$ itself) are contractible, in particular, they are acyclic.  Moreover, intersections of faces are connected unless empty.  However, the orbit space $M/T$ of a torus manifold $M$ is not necessarily a manifold with corners and even if $M/T$ is a manifold with corners, the faces of $X/T$ are not necessarily acyclic and intersections of faces are not necessarily connected.  The following holds. 

\begin{theo}[\cite{ma-pa06}] \label{mapa}
If $H^{\rm{odd}}(M)=0$, then $M/T$ is a manifold with corners.  Moreover, 
\begin{enumerate}
\item[(1)] 
every face of $M/T$ is acyclic $\Longleftrightarrow$ 
$H^{\rm{odd}}(M)=0$．
\item[(2)] any intersection of faces of $M/T$ is connected unless empty $\Longleftrightarrow$ 
$H^*(M)$ is generated by $H^2(M)$ as a ring.  
\end{enumerate}
\end{theo}

\begin{exam} \label{S2n}
When $M=S^{2n}\subset \C^n\times\R$ (see Example~\ref{exam12}(1)), 
$M/T$ is homeomorphic to an $n$-dimensional disk but $M/T$ is a manifold with corners as is shown in the following figure and intersections of faces of $M/T$ are not necessarily connected.  

\begin{figure}[h]
\begin{center}
\setlength{\unitlength}{0.6mm}
\begin{picture}(55,45)(20,-10)
\put(0,0){\includegraphics[scale=0.72]{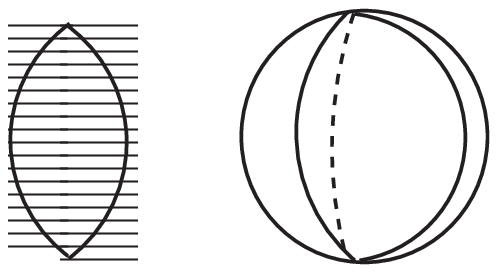}}
\put(8,-8){$S^4/T$}
\put(44,-8){$S^6/T$}
\end{picture}
\end{center}
\end{figure}
\end{exam}

As for a torus manifold $M$ with vanishing odd degree cohomology, one can describe its equivariant or ordinary cohomology neatly.  The description is similar to the toric case but unlike the toric case, $H^*(M)$ is not necessarily generated by $H^2(M)$ in general even if $H^{\rm{odd}}(M)=0$, so we need a new notion to describe $H^*(M)$, which we shall explain.  

Suppose our torus manifold $M$ satisfies $H^{\rm{odd}}(M)=0$.  Then $M/T$ is a manifold with corners and every face of $M/T$ is acyclic by Theorem~\ref{mapa}. 
Let $F$ be a face of $M/T$ and denote by $M_F$ the pullback of $F$ by the projection $q\colon M\to M/T$.  
$M_F$ is a closed $T$-invariant submanifold of $M$ and its codimension is $2\codim F$.  Therefore, 
$M_F$ represents an equivariant cycle in $M$ and its Poincar\'e dual $\tau_F$ determines an element of $H_T^{2\codim F}(M)$.  
We allow $F$ to be the entire space $M/T$ or $\emptyset$ and understand $\tau_{M/T}=1$ and $\tau_\emptyset=0$.  With this understood the following holds. 

\begin{theo}[\cite{ma-pa06}] \label{mapa2}
If $M$ is a torus manifold with $H^{\rm{odd}}(M)=0$, then 
\[
H^*_T(M)=\Z[\tau_F\mid \text{$F$ is a face of $M/T$}]/\big(\tau_F\tau_G-
\tau_{F\vee G}\sum_{\text{$E\subset F\cap G$}}\tau_E\big),
\]
where $F\vee G$ denotes the minimal face of $M/T$ which contains both $F$ and $G$, and $E$ runs over all faces in $F\cap G$.
\end{theo}

\begin{exam}
(1) When $n=2$ in Example~\ref{S2n}, the orbit space $S^4/T$ is a 2-gon.  If we label the two edges of $S^4/T$ by $F,G$ and the two vertices by $p,q$, then 
\[
H^*_T(S^4)=\Z[\tau_F,\tau_G,\tau_p,\tau_q]/\big(\tau_F\tau_G-(\tau_p+\tau_q), \tau_p\tau_q\big),
\]
where $\deg\tau_F=\deg\tau_G=2$ and $\deg\tau_p=\deg\tau_q=4$. 

(2) In the case (2) in Theorem~\ref{mapa}, let $F_1,\dots,F_m$ be the codimension one faces of $M/T$. Then any codimension $k$ face $F$ of $M/T$ is the intersection of some $k$ number of  codimension one faces $F_{i_1},\dots,F_{i_k}$, so $\tau_F=\prod_{j=1}^k\tau_{F_{i_j}}$.  Therefore, 
$H^*_T(M)$ is generated by $\tau_{F_i}$ $(i=1,\dots,m)$ and the relations among them agrees with those in the toric case. These $\tau_{F_i}$'s correspond to $\tau_i$'s in the previous section.  
\end{exam}

The set of faces of $M/T$ forms a simplicial poset under the partial order on the faces defined by:  $F\ge G$ if and only if $F\subset G$.  
A simplicial complex can be thought of as a simplicial poset under inclusion relations among simplices, and Stanley (\cite{stan91}) defined a face ring for a simplicial poset from a purely combinatorial viewpoint in such a way that when the simplicial poset is a simplicial complex, the face ring agrees with the well-known Stanley-Reisner ring of the simplicial complex.  
Theorem~\ref{mapa2} shows that the face ring of a simplicial poset has a geometrical meaning, i.e. it is the equivariant cohomology of a torus manifold.  
When $H^{\rm{odd}}(M)=0$, the ordinary cohomology ring $H^*(M)$ of $M$ is the quotient of $H^*_T(M)$ by the ideal generated by elements in Proposition~\ref{vi} as in the toric case.

\section{Characterization of face numbers of simplicial polytopes and simplicial cell spheres}

As is seen in the fundamental theorem in toric geometry, toric geometry fits well to combinatorics and interesting applications to combinatorics are known.  Among those applications, the characterization of face numbers of simplicial polytopes must have been striking as mentioned in the Introduction.  In fact, there is a similar application in toric topology, that is, the characterization of face numbers of simplicial cell spheres.  In this section, we explain the characterization and outline the proof.  One can see a difference between algebraic geometry and topology from this story.      

The convex hull of a finitely many points in $\R^n$ is called a \emph{convex polytope}.  We may assume that those points are not contained in any hyperplane and will assume it throughout this section.  Therefore our convex polytope is of dimension $n$.  A face of dimension $0$ (resp. 1) is called a vertex (resp. an edge) and a face of codimension $1$ is called a \emph{facet}.  The number of $i$-dimensional faces of a convex polytope $P$ is denoted by $f_i(P)$ (or simply $f_i$) $(0\le i\le n-1)$ and $(f_0,f_1,\dots,f_{n-1})$ is called the \emph{$f$-vector} of $P$.  

Each component $f_i$ of the $f$-vector is a positive integer but they do not take arbitrary positive integers.  For instance, since the boundary of $P$ is homeomorphic to a sphere of dimension $n-1$ and its Euler characteristic is $1+(-1)^{n-1}$, the $f$-vector must satisfy  
\begin{equation} \label{eqn:euler}
f_0-f_1+\dots+(-1)^{n-1}f_{n-1}=1+(-1)^{n-1}. 
\end{equation}
Moreover, since $P$ is of dimension $n$, the number $f_0$ of the vertices of $P$ must satisfy 
\begin{equation} \label{eqn:f_0}
f_0\ge n+1. 
\end{equation}
Like this, the $f$-vector of a convex polytope must satisfy some equalities or inequalities.  It is not difficult to characterize the $f$-vectors of convex polytopes of dimension $3$ but the characterization remains open in dimension greater than or equal to $4$ (see \cite{zieg95}).   

The integer vector $(h_0,h_1,\dots,h_n)$ defined by the equation 
\begin{equation} \label{eqn:h_vector}
\begin{split}
h_0t^n+&h_1t^{n-1}+\dots+h_n\\
:=&(t-1)^n+f_0(t-1)^{n-1}+\dots+f_{n-1}
\end{split}
\end{equation}
is called an \emph{$h$-vector}.  Clearly $h_0=1$ and $h_1=f_0-n$.  In general, the expression of $h_i$ in terms of $f_j$'s is rather complicated.  The $h$-vector contains the same information as the $f$-vector but it is often easier to treat the $h$-vector.  For instance,  (\ref{eqn:euler}) and  
(\ref{eqn:f_0}) above respectively reduce to simple forms 
\begin{equation} \label{eqn:h_vector2}
(1=)h_0=h_n, \quad (1=)h_0\le h_1. 
\end{equation}

When the finitely many points spanning the convex polytope $P$ are in a general position, every face of $P$ is a simplex; so such $P$ is called a \emph{simplicial convex polytope}. 

A convex polytope $P$ may be defined as the intersection of finitely many half spaces in $\R^n$ which is bounded.  When hyperplanes obtained as boundaries of the half spaces are in a general position, there are exactly $n$ hyperplanes meeting at each vertex of $P$.  Such $P$ is called a \emph{simple convex polytope}.   

Simplicial convex polytopes and simple convex polytopes are respectively general ones in the above two definitions of polytopes and they are dual to each other.  For example, there are five regular convex polytopes in dimension 3, where the regular tetrahedron, the regular octahedron and the regular icosahedron are simplicial while the regular tetrahedron, the regular cube and the regular dodecahedron are simple.  As is well known, the regular tetrahedron is self-dual and the regular cube (resp. dodecahedron) is dual to the regular octahedron (resp. icosahedron).  In general, if we denote the polytope dual to a convex polytopes $P$ of dimension $n$ by $P^\circ$, then $f_i(P)=f_{n-i-1}(P^\circ)$; so the characterization of face numbers of simplicial convex polytopes is equivalent to that of simple convex polytopes.   

The orbit space $X/T$ of a toric manifold $X$ is a simple convex polytope (in many cases) and the boundary complex of the dual polytope of $X/T$ is isomorphic to the underlying simplicial complex $K_X$ of the fan of $X$.  The following holds for the topology of $X$ and the combinatorics of $X/T$.  

\begin{lemm} \label{hi}
If the orbit space $X/T$ of a toric manifold $X$ is a simple convex polytope, then the $i$-th $h$-vector of $X/T$ agrees with the $2i$-th Betti number $b_{2i}(X)$ of $X$.   
\end{lemm}

Noting this fact, the characterization of the $h$-vectors of simplicial convex polytopes, which was conjectured by McMullen(\cite{mcmu71}), has been completed as mentioned in the Introduction.  

\bigskip
\noindent
{\bf $g$-Theorem.} (Billera-Lee \cite{bi-le80}, Stanley 
\cite{stan80})．\
An integer vector $(h_0,h_1,\dots,h_n)$ with $h_0=1$ is the $h$-vector of some simplicial convex polytope of dimension $n$ if and only if it satisfies the following three conditions: 
\begin{enumerate}
\item[(1)] $h_i=h_{n-i}$ $(\forall i)$ (Dehn-Sommerville equations).
\item[(2)] $(1=)h_0\le h_1\le \dots\le h_{[n/2]}$.
\item[(3)] $h_{i+1}-h_i\le (h_i-h_{i-1})^{\langle i\rangle}$ \ \ 
$(1\le i\le [n/2]-1)$.
\end{enumerate}

\bigskip
\noindent
The (1) and (2) above can be thought of as a generalization of (\ref{eqn:h_vector2}).  The meaning of the notation in (3) above is as follows.   
For positive integers $a$ and $i$, we express  
\[
a=\binom{a_i}{i}+\binom{a_{i-1}}{i-1}+\dots+\binom{a_j}{j}
\]
where $a_i>a_{i-1}>\dots>a_j\ge j\ge 1$, and define 
\medskip
\[
a^{\langle i\rangle}:=\binom{a_{i}+1}{i+1}+
\binom{a_{i-1}+1}{i}+\dots+\binom{a_j+1}{j+1}.  
\]
For example, if $a=28$ and $i=4$, then 
\[
28=\binom{6}{4}+\binom{5}{3}+\binom{3}{2},
\]
so we have 
\[
28^{\langle 4\rangle}= \binom{7}{5}+\binom{6}{4}+\binom{4}{3}=40
\]

\begin{proof}[\underbar{Outline of the proof of $g$-Theorem}]
Billera-Lee\cite{bi-le80} proved the sufficiency and Stanley\cite{stan80} proved the necessity using toric geometry.\footnote{McMullen\cite{mcmu93} proves the necessity of (1), (2), (3) in a purely combinatorial way without using toric geometry.} The argument of Stanley is as follows.  Let  $P$ be a simplicial convex polytope of dimension $n$.  Then there exists a projective toric manifold (orbifold in general) $X$ such that 
$b_{2i}(X)=h_i(P)$ (see Lemma~\ref{hi}).  Therefore, (1) follows from the Poincar\'e duality theorem, (2) follows from the hard Lefschetz theorem and (3) follows by applying Macaulay's theorem to a certain quotient ring of $H^*(X)$.  
\end{proof}

The boundary of a simplicial convex polytope of dimension $n$ gives a simplicial decomposition of a sphere of dimension $n-1$.  A sphere with a simplicial decomposition is called a \emph{simplicial sphere}.  If we define $f_i$ to be the number of $i$-simplices in the simplicial decomposition,  an $f$-vector and an $h$-vector can similarly be defined for a simplicial sphere (more generally for a simplicial complex).  Therefore it is natural to ask whether the necessity in the $g$-Theorem still holds for simplicial spheres.  It is known that condition (1) is necessary even for simplicial spheres but it is unknown whether conditions (2) and (3) are necessary for simplicial spheres.  This is a longstanding problem after $g$-Theorem is established.   

In order to make the above problem essential, we need to know the existence of a simplicial sphere which is not isomorphic to the boundary complex of a simplicial convex polytope.  In fact, it is known that there are a great many such simplicial spheres (see \cite{zieg02})．
One can also see that such an example exists from a topological point of view as follows.  We take a simplicial decomposition of a homology 3-sphere $X$ with a non-trivial fundamental group (so $X$ is not homeomorphic to the standard $3$-sphere).  The double suspension $\Sigma^2X$ of $X$ is homeomorphic to the standard $5$-sphere by Double Suspension Theorem and has a simplicial decomposition induced from that of $X$.  But $\Sigma^2X$ with the induced simplicial decomposition is not isomorphic to the boundary complex of any simplicial convex polytope because the link of a $1$-simplex obtained from the double suspension is $X$ but $X$ is not homeomorphic to the standard $3$-sphere.  

There is an object called a \emph{simplicial cell complex}.  The notion of a simplicial cell complex sits in between a simplicial complex and a cell complex.  Two simplices in a simplicial complex intersect at one simplex unless the intersection is empty. A simplicial cell complex is a cell complex whose cells are all simplices and two simplices may intersect at more than one simplex.  A simplicial cell complex determines a simplicial poset under the inclusion relation among the simplices and conversely a simplicial poset determines a simplicial cell complex.  Therefore, the notions of simplicial cell complex and simplicial poset are equivalent. 

A simplicial cell complex is called a \emph{simplicial cell sphere} if it is homeomorphic to a standard sphere.  Needless to say, a simplicial sphere is a simplicial cell sphere.  

\begin{exam}
Gluing two copies of simplices of dimension $n-1$ along their boundary by the identity map produces a simplicial cell sphere which is not a simplicial sphere.  This simplicial cell sphere is \lq\lq dual" to the boundary of the orbit space $S^{2n}/T$ in Example~\ref{S2n}.  
\end{exam}

As for a simplicial cell complex, an $f$-vector and an $h$-vector can be defined similarly and Lemma~\ref{hi} can be extended to torus manifolds with vanishing odd degree cohomology.  

The characterization of face numbers of simplicial spheres is unknown but one can characterize the face numbers of simplicial cell spheres.  The following theorem was established in \cite{stan91} except the necessity of condition (3) and the necessity was  proved in \cite{masu05}.  The proof of the necessity of condition (3) is purely algebraic but the idea of the proof stems from topology.  

\begin{theo}[\cite{masu05}, \cite{stan91}] \label{mast}
An integer vector $(h_0,h_1,\dots,h_n)$ with $h_0=1$ is the $h$-vector of some simplicial cell sphere of dimension $n-1$ if and only if it satisfies the following three conditions: 
\begin{enumerate}
\item[(1)] $h_i=h_{n-i}$ $(\forall i)$ (Dehn-Sommerville equations).
\item[(2)] $h_i\ge 0$ $(1\le i\le n-1)$.
\item[(3)] When $n$ is even, $h_{n/2}$ is even if the equality holds in (2) above for some $j$ (i.e. if $h_j=0$ for some $j$).  
\end{enumerate}
\end{theo}

\begin{proof}
Here is a brief explanation (from a topological viewpoint) why the theorem holds.  
Since $S^{2k}\times S^{2n-2k}$ $(k=1,\dots,n-1)$ and $\C P^n$ are torus manifolds with vanishing odd degree cohomology, so are their equivariant connected sums.  Their orbit spaces are manifolds with corners whose boundaries are dual to simplicial cell spheres of dimension $n-1$.  Noting that Lemma~\ref{hi} holds for torus manifolds with vanishing odd degree cohomology, we see that any $h$-vector which satisfies conditions (1), (2) and (3) can be realized as $h$-vectors of simplicial cell spheres of dimension  $n-1$.       
  
The idea of the proof of the necessity is as follows.  Suppose that a simplicial cell sphere $K$ of dimension $n-1$ is dual to the boundary of the orbit space $M/T$ of some torus manifold $M$ with vanishing odd degree cohomology.  Then $h_i(K)=b_{2i}(M)$. Therefore, condition (1) follows from Poincar\'e duality and condition (2) follows from Lemma~\ref{hi} extended to torus manifolds.  The problem is to prove condition (3).   

If $h_j(K)=0$ for some $j$, then $b_{2j}(M)=0$, in particular, $H^{2j}(M;\Z/2)=0$.  On the other hand, one can see that the 
Stiefel-Whitney class of $M$ is of the form 
$$w(M)=\prod_{i=1}^m(1+\mu_i)\quad (\mu_i\in H^2(M;\Z/2)),$$
where $\mu_i$ is the Poincar\'e dual to the characteristic submanifold $M_i$.  
These imply that $w_{2k}(M)=0$ $(k\ge j)$ because $H^{2j}(M;\Z/2)=0$ and $w(M)$ is a polynomial in the degree two elements $\mu_i$'s.  In particular, the top Stiefel-Whitney class $w_{2n}(M)$ vanishes and this means that the Euler characteristic $\chi(M)$ of $M$ is even.  This together with (1) and the identity  
$$\chi(M)=\sum_{i=1}^nb_{2i}(M)=\sum_{i=1}^nh_i(K)$$
implies condition (3).  
\end{proof}

If (the realization of) a simplicial cell complex $K$ is homeomorphic to a manifold $N$, then $K$ is called a simplicial cell decomposition of $N$.  
One can ask the characterization of $h$-vectors of simplicial cell decompositions for not only a sphere but also other manifolds (such as real projective spaces and disks).  When $N$ is a closed manifold, those $h$-vectors might be characterized by three types of conditions like Theorem~\ref{mast}, but there is no complete answer to this question except a sphere (see \cite{masu04} for some trial)\footnote{Added in the English translation:  Recently S. Murai (\cite{mura10}, \cite{mura11}) has characterized the $h$-vectors of simplicial cell decompositions when $N$ is a real projective space, a disk or a product of two spheres.}. Dehn-Sommerville equations hold for any closed manifold with a little modification as follows.         

\begin{theo}[p.74 in \cite{stan96} or Theorem 7.44 in \cite{bu-pa02}] \label{theo:gdse}  
The $h$-vector $(h_0,\dots,h_n)$ of a simplicial cell decomposition of a closed manifold $N$ of dimension $n-1$ satisfies the following:
\[
h_{n-i}-h_i=(-1)^i\big(\chi(N)-\chi(S^{n-1})\big)\binom{n}{i}\qquad (1\le \forall i\le n).
\]
\end{theo}

The notion of a simplicial cell complex is the one which weakens the notion of a simplicial complex but it seems that simplicial cell complexes have not been so much studied.  One of the reason would be that its relation to geometry was weak.  However, simplicial cell complexes are related to toric topology like simplicial convex polytopes are related to toric geometry.  Therefore, one can expect that results or ideas in topology can be applied to the study of simplicial cell complexes.  Simplicial cell complexes should be studied more, which would also improve our understanding of  simplicial convex polytopes and simplicial complexes.

\section{Counting lattice points}

In this section we will discuss another well-known application of toric geometry to combinatorics, which is counting lattice points in a lattice convex polytope through a moment map.  It turns out that this story in toric geometry can also be generalized to toric topology.  In this generalized setting, moment map still exists and its image is a polytope in some sense but not necessarily convex and may have self-intersection. We will see that results on counting lattice points in a lattice convex polytope hold in this general setting with appropriate modification.    

By \eqref{HomST} we have an identification 
\[
\Hom(T,S^1)=H^2(BT)
\]
and a homomorphism from $T$ to $S^1$ determines a lattice point in the dual $\Lie(T)^*$ to the Lie algebra $\Lie(T)$ of $T$ through differential.  
Therefore, we have 
\begin{equation} \label{LieT*}
\Lie(T)^*=H^2(BT;\R)\supset H^2(BT)=\Hom(T,S^1),
\end{equation}
so a point in the lattice $H^2(BT)$ can be thought of as a complex one-dimensional representation of $T$ and we will denote by $t^u$ the complex one-dimensional representation of $T$ corresponding to $u\in H^2(BT)$.   

To an ample $T$-equivariant complex line bundle $L$ over a toric manifold $X$, a moment map 
$$\Phi_L\colon X\to \Lie(T)^*=H^2(BT;\R)$$ 
is associated and the image $\Phi_L(X)$ is a lattice convex polytope (i.e. a convex polytope with lattice points as vertices), see the following figure.  The dimension of $\Phi_L(X)$ agrees with $\dim_\C X=n$.  

\begin{figure}[h]
\begin{center}
\setlength{\unitlength}{0.6mm}
\begin{picture}(25,35)(10,4)
\put(0,0){\includegraphics[scale=0.9]{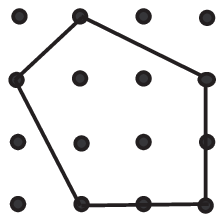}}
\end{picture}
\end{center}
\end{figure}

The moment map image $\Phi_L(X)$ lies in $H^2(BT;\R)$ while the space in which the fan $\Delta_X$ of $X$ lies was $H_2(BT;\R)$ which is dual to $H^2(BT;\R)$.  The relation between $\Phi_L(X)$ and $\Delta_X$ is as follows.  
Since the complex line bundle $L\to X$ is $T$-equivariant, one can apply the Borel construction to the equivariant line bundle and obtains a complex line bundle $ET\times_T L\to ET\times_T X$.  We denote by $c_1^T(L)$ the first Chern class of the resulting line bundle and call it the equivariant first Chern class of $L$.  Since $c_1^T(L)$ is an element of $H^2_T(X)$, one can express it as 
\begin{equation} \label{c1}
c_1^T(L)=\sum_{i=1}^ma_i\tau_i \quad (a_i\in\Z)
\end{equation}
by Proposition~\ref{ring}, and using the integers $a_i$ above we have  
\begin{equation} \label{PhiL}
\Phi_L(X)=\{u\in H^2(BT;\R)\mid \langle u,v_i\rangle \le a_i\quad(i=1,\dots,m)\}
\end{equation}
where $\langle\ ,\ \rangle$ denotes the usual paring between cohomology and homology. 

As is well-known, the set $H^0(X,L)$ of holomorphic sections of the line bundle $L\to X$ forms a finite dimensional complex representation space of $T$.  The following theorem describes this representation space in terms of the moment map $\Phi_L$.

\begin{theo}[See \cite{oda88}, \cite{fult93} for example]　\label{H0} 
$\displaystyle{H^0(X;L)=\sum_{u\in \Phi_L(X)\cap H^2(BT)}t^u}.$
\end{theo}

It follows from the ampleness of the line bundle $L$ that $H^i(X;L)=0$ $(i>0)$, so that the left hand side in the theorem above agrees with the equivariant Riemann-Roch index $RR^T(X;L)$.  In particular, if we forget the $T$-action, we obtain the following.  

\begin{coro} \label{latt}
The number $\#(\Phi_L(X))$ of lattice points in the lattice polytope $\Phi_L(X)$ agrees with the Riemann-Roch number $RR(X;L)$ of $L$. 
\end{coro} 

Theorem~\ref{H0} and Corollary~\ref{latt} are the key connecting toric geometry and the problem of counting lattice points in a lattice convex polytope.  We shall explain this story with some examples.   

\begin{theo}[Ehrhart] \label{Ehrh}
Let $P$ be a lattice convex polytope of dimension $n$ and let $q$ be a positive integer.  Then the number $\#(qP)$ of lattice points in the lattice polytope $qP=\{qx\mid x\in P\}$ obtained by dilating $P$ $q$ times is a polynomial in $q$ of degree $n$ with rational coefficients.  Moreover, if we express $\#(qP)=\sum_{i=0}^na_i(P)q^i$ $(a_i(P)\in \Q)$, then $a_n(P)$ is the volume of $P$ and $a_0(P)=1$.  
\end{theo}

\begin{proof}[Outline of the proof using Corollary~\ref{latt}] 
We may assume that $P$ is in $\Lie(T)^*=H^2(BT)$ and find a toric manifold $X$ together with an ample $T$-line bundle $L$ over $X$ whose moment map $\Phi_L\colon X\to H^2(BT)$ has $P$ as the image.  The $q$-fold tensor product $L^q$ is again an ample $T$-line bundle and its associated moment map $\Phi_{L^q}$ has $qP$ as the image.  Therefore, it follows from Corollary~\ref{latt} and Hirzebruch-Riemann-Roch Theorem that   
\begin{equation} \label{qP}
\#(qP)=\#(\Phi_{L^q}(X))=\langle e^{c_1(L^q)}\Td(X),[X]\rangle,
\end{equation}
where $c_1(\ )$ denotes the first Chern class, $\Td(X)$ the Todd class of $X$, and $[X]$ the fundamental class of $X$. 
Since we have $e^{c_1(L^q)}=e^{qc_1(L)}=\sum_{k=0}^nq^kc_1(L)^k/{k!}$, the right hand side of \eqref{qP} is a polynomial in $q$ of degree $n$ with rational coefficients and 
\[
a_0(P)=\text{Todd genus of $X$},\quad a_n(P)=\langle c_1(L)^n/n!,[X]\rangle.  
\]
Here $a_0(P)=1$ because the Todd genus of a toric manifold is $1$.  Finally, it is not difficult to see that $\langle c_1(L)^n/n!,[X]\rangle$ agrees with the volume of $P$. 
\end{proof}

The coefficients $a_i(P)$ are invariants of $P$.  $a_n(P)$ is the volume of $P$ as in Theorem~\ref{Ehrh} and $a_{n-1}(P)$ is known to be half of the relative volume $\vol(\partial P)$ of the boundary $\partial P$ of $P$ which was also proved by Ehrhart. Here the relative volume $\vol(\partial P)$ is sum of the volume of facets $F$ of $P$ and the volume of $F$ is defined by normalizing the volume of a minimal lattice cube of dimension $n-1$ in the hyperplane containing $F$ to be $1$.  The coefficient $a_{n-2}(P)$ is explicitly described in terms of $P$ but complicated (\cite{pomm93}, 
\cite{zieg02}).  The description of the coefficients $a_i(P)$ except $i=0, n-2, n-1, n$ in terms of $P$ seems unknown. 

When $P$ is 2-dimensional, that is, a convex polygon, Theorem~\ref{Ehrh} with $q=1$ implies 
\[
\#(P)=\Area(P)+\frac{1}{2}\#(\partial P)+1
\]
because $a_{n-1}(P)=\frac{1}{2}\vol(\partial P)$ reduces to the number $\#(\partial P)$ of lattice points on the boundary of $P$ when $n=2$.  If we denote by $\#(\Int P)$ the number of lattice points in the interior of $P$, then the identity above can be written as 
\begin{equation} \label{Pick}
\Area(P)=\#(\Int P)+\frac{1}{2}\#(\partial P)-1
\end{equation}
because $\#(P)=\#(\Int P)+\#(\partial P)$.  The identity \eqref{Pick} is well known as Pick's formula. 
Pick's formula holds even for concave polygons but the concave case cannot be proved using toric geometry.  This implies existence of a theory extending toric geometry and one can say that toric topology is the desired theory.  In fact, Pick's formula can be proved in full generality including concave polygons.  Furthermore, even for lattice polygons with self-intersections like the right figure in Figure~\ref{star}, both sides in \eqref{Pick} can naturally be defined and Pick's formula can be generalized to these polygons.  In this generalization, the constant term at the right hand side of \eqref{Pick} is not necessarily $1$ and the geometrical meaning of the constant term becomes clear.    

\begin{figure}[h]
\begin{center}
\setlength{\unitlength}{0.6mm}
\begin{picture}(55,50)(10,5)
\put(-30,0){\includegraphics[scale=0.9]{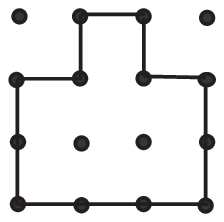}}
\put(40,0){\includegraphics[scale=0.9]{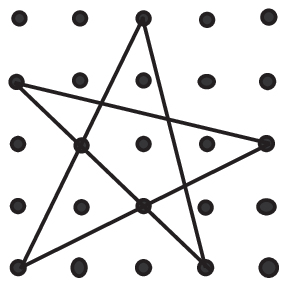}}
\end{picture}
\end{center}
\caption{}\label{star}
\end{figure}


We shall explain how the story above can be developed in toric topology.  When $M$ is a torus manifold and $L$ is a $T$-line bundle, the space of holomorphic sections $H^0(M;L)$ does not make sense but the equivariant Riemann-Roch index $RR^T(M;L)$ and the Riemann-Roch number $RR(M;L)$ can be defined.  $RR^T(M;L)$ is an element of the representation ring $R(T)$ of $T$ and $RR(M;L)$ is an integer.  Moreover, the moment map $\Phi_L$ still exists but its image $\Phi_L(M)$ is not necessarily convex like \eqref{PhiL}.  It can be concave and can have many overlapping as in Figure~\ref{star}.  We can associate an arrangement of hyperplanes 
$$H_i:=\{u\in H^2(BT;\R)\mid \langle u,v_i\rangle=a_i\}\quad (i=1,\dots,m)$$
to the multi-fan $\Delta_M=(K_M,\{v_i\}_{i=1}^m,w)$ of $M$ using \eqref{c1}, and the facets of $\Phi_L(M)$ are contained in this arrangement.  This observation leads us to consider a pair $\P=(\Delta_M,\{H_i\}_{i=1}^m)$ of the multi-fan $\Delta_M$ and the hyperplane arrangement $\{H_i\}_{i=1}^m$.   Such $\P$ is called a \emph{multi-polytope} (\cite{ha-ma03}).  A lattice convex polytope determines an (ordinary) fan and can be regarded as a multi-fan.  

The hyperplane arrangement $\{H_i\}_{i=1}^m$ of a multi-polytope $\P$ divides the entire space $H^2(BT;\R)$ into some regions and a function $\m$ which assigns an integer to each region can be defined using the datum $w$ incorporated in the definition of a multi-fan.  We omit the precise definition of $\m$ but the geometrical meaning of $\m(u)$ is the (signed) multiplicity which measures how many times $\Phi_L(M)$ covers $u$, or the winding number (or mapping degree) which measures how many times the boundary of $\Phi_L(M)$ goes around $u$.  For example, as for the left polygon in Figure~\ref{star}, $\m(u)=0$ if $u$ is outside the polygon, $\m(u)=0$ if $u$ is inside the polygon.  As for the right star shaped polygon in Figure~\ref{star}, $\m(u)=\pm 2$ if $u$ is inside the pentagon, $\m(u)=\pm 1$ if $u$ is outside the pentagon but is in the interior of the star-shaped polygon, and $\m(u)=0$ if $u$ is outside the star-shaped polygon.  A rather delicate argument is necessary in any case when $u$ is on the segments of the polygon.  

In toric geometry, the function $\m$ takes $1$ on $\Phi_L(M)$ and $0$ outside of $\Phi_L(M)$.  Theorem~\ref{H0} is generalized as follows. 

\begin{theo}[\cite{ka-to93}, \cite{gr-ka98}, \cite{masu99}] 
\label{theo:twis}\ 
$\displaystyle{RR^T(M,L)=\sum_{u\in \Phi_L(M)}\m(u)t^u\in R(T)}$. 
\end{theo}

A generalization of Corollary~\ref{latt} from a topological viewpoint can be obtained from the theorem above by forgetting the $T$-action.  Moreover, Theorem~\ref{Ehrh} by Ehrhart and Pick's formula \eqref{Pick} can be generalized (\cite{ha-ma03}, \cite{masu99}).  The point in these generalizations is to count lattice points with the integer valued function $\m$.  

One can also generalize some other results on counting lattice points in a convex lattice polytope to our setting.  Those are a generalization of Euler-Maclaurin formula which counts lattice points with a polynomial weight function (\cite{ka-st-we05}, \cite{ka-st-we07}),   Khovanskii-Pukhlikov formula (\cite{kh-pu93}, \cite{guil94}) which finds the number of lattice points in a convex polytope through volumes of variation of $P$, and the twelve-point theorem (\cite{po-ro00}) which says that sum of the number of lattice points on the boundary of a lattice convex polygon with only one interior lattice point and that of its dual lattice convex polygon is always 12.  These results hold for multi-polytopes with appropriate modification.

\section{Moment-angle complex and coordinate subspace arrangement}

In Section 2, we explained the correspondence from toric manifolds to fans but since this correspondence is bijective (Theorem~\ref{fund}), there is also a correspondence from fans to toric manifolds, in other words, a construction of toric manifolds from fans.  Complex projective space $\C P^n$ is a typical example of toric manifolds and there are (at least) the following three ways to construct it: 
\begin{enumerate}
\item[(1)] gluing local charts $\C^n$,
\item[(2)] $(\C^{n+1}\backslash\{0\})/\C^*$
\item[(3)] $S^{2n+1}/S^1$
\end{enumerate}
and these three constructions can respectively be extended to general toric manifolds.\footnote{Precisely speaking, toric manifolds need to be projective in (3).} See \cite{fult93}, \cite{oda88} for (1).  Construction (3) is called symplectic quotient, see \cite{guil94}.  We shall explain construction (2) below (\cite{cox95})．

Remember that the fan $\Delta_X$ of a toric manifold consists of two data.  One is a simplicial complex 
\[
K_X=\{ I\subset [m] \mid \bigcap_{i\in I}X_i\not=\emptyset\}.
\]
which encodes intersections of characteristic submanifolds $X_i$ $(i=1,\dots,m)$ of $X$.  The other is the set of vectors $v_i\in H_2(BT_\C)$ which represents the $\C^*$-subgroup of $T_\C$ fixing $X_i$ pointwise.  

To a subset $\sigma=\{i_1,\dots.i_k\}$ of the vertex set $[m]$ of a simplicial complex $K$, we define 
\[
L_\sigma:=\{ (z_1,\dots,z_m)\in\C^m\mid z_{i_1}=\dots=z_{i_k}=0\}
\]
and conisder
\begin{equation} \label{UK}
U(K):=\C^m\backslash \bigcup_{\sigma\notin K}L_\sigma.
\end{equation}
When the fan $\Delta=(K,\{v_i\}_{i=1}^m)$ is complete and non-singular, the toric manifold $X(\Delta)$ corresponding to the fan $\Delta$ is obtained as the orbit space 
$$U(K)/\ker \V=:X(\Delta)$$
of $U(K)$ by the kernel $\ker \V$ of a homomorphism 
$$\V:=\prod_{i=1}^m \lambda_{v_i}\colon (\C^*)^m\to T_\C$$
Even if the simplicial complex $K$ is same, a different choice of vectors $\{v_i\}$ will produce a different toric manifold $X(\Delta)$. For instance, there are infinitely many Hirzebruch surfaces but   
$U(K)=(\C^2\backslash \{0\})\times (\C^2\backslash \{0\})$ in any case. 

It seems that the topology of $U(K)$ has not been studied so much but the above construction shows that the topology of toric manifolds can be used to study the topology of $U(K)$.

One can see from the definition \eqref{UK} of $U(K)$ that $U(K)$ can also be written as 
\[
U(K)=\bigcup_{\sigma\in K}\{(z_1,\dots,z_m)\in \C^m\mid  z_j\not=0 \text{ for $j\notin \sigma$} \}.
\]
Therefore, noting that $(D^2,S^1)$ is a deformation retract of $(\C,\C^*)$, we see that 
\[
\mathcal{Z}_K:=\bigcup_{\sigma\in K}\{(z_1,\dots,z_m)\in (D^2)^m\mid z_j\in S^1\text{ for $j\notin \sigma$} \}
\]
is a deformation retract of $U(K)$.  Although $U(K)$ is non-compact, 
$\mathcal{Z}_K$ is compact.  For example, when $K$ is the boundary complex of a simplex of dimension $n-1$, 
$U(K)=\C^n\backslash\{0\}$ and $\mathcal{Z}_K=S^{2n-1}$. 
The space $\mathcal{Z}_K$ can be thought of as a level set of a moment map and when $U(K)=(\C^2\backslash\{0\})\times (\C^2\backslash\{0\})$ which appears in the construction of Hirzebruch surfaces, $\mathcal{Z}_K=S^3\times S^3$. 

The spaces $U(K)$ and $\mathcal{Z}_K$ can be defined for any simplicial complex $K$ and 
$\mathcal{Z}_K$ is a manifold when the geometric realization of $K$ is a sphere but otherwise $\mathcal{Z}_K$ is not a manifold in general.  
The space $\mathcal{Z}_K$ is called a \emph{moment-angle complex} and Buchstaber-Panov\cite{bu-pa02} used topological technique to study the cohomology ring of $\mathcal{Z}_K$ and hence $U(K)$. 
Panov \cite{pano06} also discusses relation of $\mathcal{Z}_K$ to a Kempf-Ness set known in geometric invariant theory.

\smallskip
Bosio-Meersserman \cite{bo-me04} construct many examples of $U(K)$ whose homology have many torsion elements when $K$ is the boundary complex of a simplicial complex.  This means that the topology of $U(K)$ is complicated (or rich) in general.  Moreover, they produce many examples of non-K\"ahler complex manifolds using $\mathcal{Z}_K$.  Those complex manifolds can be explicitly described as complete intersections of real quadric hypersurfaces in $\C^m$.  

In general, it is difficult to determine the homotopy type of $U(K)$ but completely determined in some cases.  

\begin{theo}[\cite{gr-th06}] \label{grth}
If $K$ is the $(k-1)$-skeleton of a simplex of dimension $m-1$ where $1\le k\le m-1$, then 
\[
U(K)\simeq \bigvee_{j=k+1}^m \binom{m}{j}\binom{j-1}{k}S^{k+j}.
\]
\end{theo}

For the simplicial complex $K$ in the theorem above, $U(K)$ is the complement of all codimension $k+1$ coordinate subspaces in $\C^m$.

\section{Rigidity problems for toric manifolds}

Classification of toric manifolds reduces to classification of corresponding fans by the fundamental theorem in toric geometry.  
But a toric manifold is an equivariant object.  So a natural question is \lq\lq If two toric manifolds are non-equivariantly isomorphic as algebraic varieties, then how the corresponding fans are related?".  The answer is known as follows.  If two toric manifolds are non-equivariantly isomorphic, then they are weakly equivariantly isomorphic and hence there is an automorphism of the lattice which sends one fan to the other.  Therefore, the isomorphism classes of toric manifolds as algebraic varieties agree with the isomorphism classes of the corresponding fans.  

\begin{exam} \label{hirz}
Let $\gamma$ be the Hopf line bundle over $\C P^1$ and let $a$ be an integer.  A Hirzebruch surface $H_a$ is obtained as projectivization of the Whitney sum of $\gamma^a$ and the trivial line bundle.  Then the fan of $H_a$ with a suitable $(\C^*)^2$-action is as in the following figure.  The fans of $H_a$ and $H_{-a}$ map to each other through the reflection with respect to the $x$-axis.  Therefore it turns out that $H_a$ and $H_{-a}$ are isomorphic.  In fact, it is known that two Hirzebruch surfaces $H_a$ and $H_b$ are isomorphic if and only if $a=\pm b$.   

\begin{figure}[h]
\begin{center}
\setlength{\unitlength}{1mm}
\begin{picture}(20,0)(55,30)
\put(50,0){\includegraphics[scale=1]{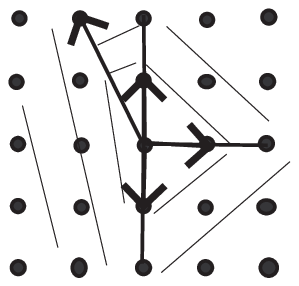}}
\put(60,-5){\large{$(a=2)$}}
\end{picture}
\end{center}
\end{figure}

\end{exam}

\vspace{3.5cm}

As for the classification of toric manifolds as algebraic varieties, we refer the reader to \cite{oda88}, \cite{sato06} and the references therein.  

The classification of toric manifolds reduces to the classification of their fans as explained above, but it also reduces to the classification of their equivariant cohomology as algebras over $H^*(BT)$.  In other words, equivariant cohomology distinguishes toric manifolds (\cite{masu08}).  This leads us to ask what properties of toric manifolds an ordinary cohomology distinguishes.  
One can observe that the diffeomorphism types of toric manifolds are distinguished by their cohomology rings in some cases.  For instance, the Hirzebruch surfaces $H_a$ and $H_b$ in Example~\ref{hirz} are diffeomorphic if and only if the integers $a$ and $b$ are congruent modulo $2$, and one can see that this is equivalent to their cohomology rings being isomorphic.  Based on these observations, we posed the following problem in \cite{ma-su08}.

\medskip
\noindent
{\bf Cohomological rigidity problem.} If two toric manifolds have isomorphic cohomology rings, then are they diffeomorphic (or homeomorphic)? 

\medskip

At first glance, the reader may not believe that the problem is affirmative but no counterexamples are known so far.  Since toric manifolds are simply connected, it follows from Freedmann's theorem on 4-manifolds that toric manifolds of complex dimension 2 are homeomorphic if their cohomology rings are isomorphic.\footnote{Added in the English translation. We can see that they are even diffeomorphic if we use the classification result on toric manifolds of complex dimension 2.}  We also have a partial affirmative solution to the cohomological rigidity problem in \cite{ch-ma-su08-2}.    
For example, the following holds.  

\begin{theo}[\cite{ch-ma-su08-2}] 
A toric manifold whose cohomology ring is isomorphic to that of $\prod_{i=1}^k\C P^{n_i}$ is diffeomorphic to $\prod_{i=1}^k\C P^{n_i}$, where $k$ and $n_i$ are arbitrary natural numbers.  
\end{theo}

Many invariants treated in topology such as cohomology theory, K-theory and bordism theory are homotopy invariants, and surgery theory classifies manifolds of a fixed homotopy type up to diffeomorphism or homeomorphism.  Therefore, it may be more natural to consider the following problem which is weaker than the cohomological rigidity problem.     

\medskip
\noindent
{\bf Homotopical rigidity problem.} If two toric manifolds are homotopy equivalent, then are they diffeomorphic (or homeomorphic)? 
\medskip

It follows from surgery theory that there are infinitely many closed smooth manifolds which are homotopy equivalent to $\C P^n$ $(n\ge 3)$ but not diffeomorphic to each other (\cite{hsia66} for $n\ge 4$ and \cite{mo-ya66}, \cite{wall66} for $n=3$), but $\C P^n$ is the only toric manifold among them.  Even for a toric manifold $X$ different from $\C P^n$, one can also see that there are infinitely many closed smooth manifolds which are homotopy equivalent to $X$ but not diffeomorphic to each other by applying the technique developed in \cite{hsia66} for $n\ge 4$ and using the results of \cite{wall66} and \cite{jupp73} for $n=3$.  

One can consider problems related to the above rigidity problems.  We shall state two of them.  Related problems can be found in \cite{ma-su08}. 

\medskip
\noindent
{\bf Problem.} Is any cohomology ring isomorphism between two toric manifolds induced by a diffeomorphism between them? \footnote{Added in the English translation.  This is not true in general.  It was known that some cohomology automorphism of $\C P^2\# 10\overline{\C P^2}$ cannot be induced by any diffeomorphism, see R. Friedman and J. W. Morgan, \emph{On the diffeomorphism types of certain algebraic surfaces. I}, 
J. Diff. Geom. 27 (1988), 297--369.}

\medskip
\noindent
{\bf Problem.} Is the decomposition of a toric manifold into a cartesian product of (indecomposable) toric manifolds unique in the smooth category?
Namely, if two cartesian products $\prod_{i=1}^k X_i$ and $\prod_{j=1}^\ell Y_j$ are diffeomorphic, where $X_i$ and $Y_j$ are indecomposable toric manifolds, then $k=\ell$ and $X_i$ and $Y_i$ are diffeomorphic for each $i$ if we change the  order of the products.   

\medskip
The orbit space of a toric manifold by the compact torus $T$ is a simple convex polytope (in most cases).   From this viewpoint, Choi-Panov-Suh\cite{ch-pa-su08} discusses an analogous rigidity problem for simple convex polytopes. \footnote{Added in the English translation: See \cite{ch-ma-su11} for the development on the rigidity problems after the original Japanese version of this article was written.}

\section{Related topics}

Finally we will discuss related topics briefly.  

\medskip
\noindent
(1) {\bf Real toric manifolds.}
A real toric manifold can be defined if we take $\R$ as the ground field instead of $\C$.  If $X$ is a toric manifold, then the set $\XR$ of real points in $X$ is a real toric manifold.  $X$ has an involution taking the complex conjugation on each local coordinate and $\XR$ is the fixed point set of the involution.  Since $X$ has the restricted action of the torus $T=(S^1)^n$, $\XR$ has an action of a 2-torus group $\TR=(S^0)^n$ and $X/T=\XR/\TR$.  

\begin{exam} \label{rBott}
(1) When $X=\C P^n$, $\XR=\R P^n$.  The orbit space $X/T=\XR/\TR$ is an $n$-simplex. 

(2) $If X$ is a Bott manifold $B_n$ in Example~\ref{Bott}, then $\XR$ is the top manifold of an iterated $\R P^1$-bundle obtained by replacing $\C$ by $\R$.  This manifold is called a \emph{real Bott manifold}.  Each $\R P^1$-bundle in the tower becomes trivial if we take a suitable double cover, so $\XR$ has an $n$-dimensional torus as a $2^n$-cover so that $\XR$ provides an example of a flat Riemannian manifold.  The orbit space    
$X/T=\XR/\TR$ is an $n$-cube. 
\end{exam}

Many arguments developed for toric manifolds hold for $\XR$ with suitable modification.  One major difference is that a toric manifold is simply connected while a real toric manifold is never simply connected and provides many examples of aspherical manifolds (a manifold whose universal cover is contractible) like Example~\ref{rBott} (2).   

It is difficult to determine the ring structure of $\XR$ with integer coefficient, but the following holds
\[
H^*(\XR;\Z/2)\cong H^{2*}(X;\Z/2)=H^{2*}(X;\Z)\otimes \Z/2,
\]
so the cohomology ring of $\XR$ with $\Z/2$-coefficient is understood if we combine the above with Theorem~\ref{tcring}.  
Therefore it seems natural to consider the cohomological rigidity problem for real toric manifolds with $\Z/2$-coefficient.  This rigidity problem is affirmatively solved for real Bott manifolds (\cite{ka-ma08}, \cite{masu08-2}\footnote{Added in the English translation. The paper \cite{masu08-2} was improved to \cite{ch-ma-ou10}.}) while there is a counterexample to the rigidity problem (\cite{masu08-3}).  However, the counterexample does not provide a counterexample to the homotopical rigidity problem.  It is also proved in \cite{masu08-2}\footnote{Added in the English translation.  See also \cite{ch-ma-ou10}.} that the two problems mentioned at the end of the last section are affirmative in the category of real Bott manifolds.  

A quasitoric manifold $M$ introduced in Section 4 also admits an involution and its fixed point set $\MR$ is a small cover (\cite{da-ja91}).  A small cover can be thought of as a topological version of a real toric manifold.  Similarly to the toric case, $\MR$ admits an action of the 2-torus group $\TR$ and $Q=M/T=\MR/\TR$.  By the definition of a quasitoric manifold, $Q$ is a simple convex polytope of dimension $n$.  A small cover over $Q$ can be obtained by gluing $2^n$ number of copies of $Q$ along their facets and all the copies appear at each vertex of $Q$.   
When $Q$ is an $n$-cube as in Example~\ref{rBott} (2), a small cover over $Q$ admits a flat Riemannian metric if we realize $Q$ in $\R^n$ in the standard way.  Similarly, if one can realize $Q$ in a hyperbolic space with right angle at each vertex, then a small cover over $Q$ becomes a hyperbolic manifold.  The dodecahedron and the 120 cell can be realized in a hyperbolic space with right angle at each vertex and small covers over them are classified from this point of view in \cite{ga-sc02}.  Small covers of dimension 3 are studied in \cite{ba-lu06}, \cite{lu-yu08}.

\medskip
\noindent
(2) {\bf Quaternionic version of toric manifolds.} 
The family of real toric manifolds contains $\R P^n$ and that of toric manifolds contains $\C P^n$.  Then it is natural to expect that there is a nice family of manifolds which contains the quaternion projective spaces $\H P^n$.  R. Scott\cite{scot95} has developed a quaternionic version of quasitoric manifolds for a family of manifolds containing $\H P^n$ by replacing the group $S^1$ by the group $S^3$, but his trial seems not so successful because of the non-commutativity of $S^3$.     

On the other hand, there is a family of manifolds which can be regarded as a quaternionic version of toric geometry in some sense although the family does not contain $\H P^n$.  Any projective toric manifold can be obtained as the symplectic quotient of $\C^m$ by an action of a torus.  Applying this idea to quaternions, one obtains hypertoric manifolds (or toric hyperK\"ahler manifolds).  A hypertoric manifold has three complex structures corresponding to the $i,j,k$ in the quaternions and is of real $4n$-dimension.  A hypertoric manifold corresponds to a hyperplane arrangement in $\R^n$ like a toric manifold corresponds to a fan, and the cohomology ring of a hypertoric manifold can be described in terms of the associated hyperplane arrangement like the cohomology ring of a toric manifold is described in terms of the associated fan (\cite{konn00}).  We refer the reader to \cite{konn08}, \cite{prou08} for details on hypertoric manifolds.

\medskip
\noindent
(3) {\bf GKM theory.}
If $M$ is a torus manifold with $H^{\rm{odd}}(M)=0$, then the restriction  
\begin{equation} \label{rest}
H^*_T(M)\to H^*_T(M^T)=\bigoplus_{p\in M^T}H^*_T(p)=\bigoplus_{p\in M^T}H^*(BT)
\end{equation}
to the fixed point set $M^T$ is injective but not surjective. Therefore, it is important to determine the image of the restriction map above.  It turns out that the image can be determined by the tangential $T$-representation at each fixed point and 2-spheres fixed pointwise under some codimension one subgroups of $T$.   

The above result is known to hold for a wider class of manifolds with torus actions than torus manifolds.  
The dimension of the acting torus $T$ can be smaller than half of the dimension of the manifold $M$ if the fixed point set $M^T$ is isolated and weights of the tangential $T$-representation at a fixed point are pairwise linearly independent (\cite{go-ko-ma98}, \cite[Chapter 11]{gu-st99}, \cite{fr-pu07}).  
We regard each fixed point as a vertex and a 2-sphere (having two $T$-fixed points) fixed pointwise under some codimension one subgroup of $T$ as an edge. Then we obtain an $n$-valent graph.  Moreover, we attach a weight determined by the tangential representations to each edge, so that we obtain an $n$-valent graph with a direction assigned to each edge.  This graph is called a \emph{GKM graph} named after Goresky-Kottwitz-MacPherson.  According to \cite{go-ko-ma98}, the equivariant cohomology of $M$ is determined by the GKM graph associated with $M$.  Grassmannians and flag manifolds are not toric manifolds in general but they have torus actions satisfying the conditions mentioned above and Schubert calculus on these manifolds can be developed in terms of the associated GKM graphs.  We refer the reader to \cite{gu-za01}, \cite{gu-za03}, \cite{kn-ta03}, \cite{macp07}, \cite{tymo08} on this topic.




\end{document}